\def\version{02/02/2012 \quad Version 5
\hfill
\href{http://arxiv.org/abs/1005.5370}{arXiv:1005.5370}}
\documentclass[a4paper,11pt]{amsart}
\usepackage{amssymb,amscd,amsxtra}
\usepackage{fullpage}
\usepackage[all,cmtip]{xy}
\usepackage{mathrsfs}
\usepackage{hyperref}

\theoremstyle{plain}
\newtheorem{thm}{Theorem}[section]

\newtheorem{lem}[thm]{Lemma}
\newtheorem{prop}[thm]{Proposition}
\newtheorem{cor}[thm]{Corollary}

\theoremstyle{definition}
\newtheorem{rem}[thm]{Remark}
\newtheorem{defn}[thm]{Definition}
\newtheorem{ex}[thm]{Example}

\numberwithin{equation}{section}

\def\ie{\emph{i.e.}}

\def\:{\colon}
\def\.{\cdot}

\def\<{\left\langle}
\def\>{\right\rangle}
\def\({\left(}
\def\){\right)}
\def\ph#1{\phantom{#1}}
\def\epsilon{\varepsilon}
\def\phi{\varphi}

\def\leq{\leqslant}
\def\geq{\geqslant}

\def\lra{\longrightarrow}
\def\Lra{\Longrightarrow}

\def\bar#1{\overline{#1}}

\def\tilde#1{\widetilde{#1}}
\def\iso{\cong}

\DeclareMathOperator{\Ker}{Ker}\renewcommand{\ker}{\Ker}

\DeclareMathOperator{\rank}{rank}

\def\F{\mathbb{F}}

\def\H{\mathbb{H}}
\def\k{{\boldsymbol{k}}}
\def\Q{\mathbb{Q}}

\def\R{\mathbb{R}}

\def\Z{\mathbb{Z}}

\DeclareMathOperator{\End}{End}
\DeclareMathOperator{\Ext}{Ext}

\DeclareMathOperator{\Hom}{Hom}

\DeclareMathOperator{\THH}{THH}
\DeclareMathOperator{\Tor}{Tor}
\DeclareMathOperator{\Map}{Map}
\def\id{\mathrm{id}}

\def\op{\mathrm{o}}
\DeclareMathOperator{\Br}{Br}

\DeclareMathOperator{\Syl}{Syl}
\DeclareMathOperator{\Gal}{Gal}

\DeclareMathOperator{\Az}{Az}
\def\nr{\mathrm{nr}}
\def\Fc{\bar{\F}}
\def\Enr{E^\nr}
\def\Knr{K^\nr}

\DeclareMathOperator{\Tot}{Tot}

\begin{document}
\title[Brauer groups for $S$-algebras]
{Brauer groups for commutative $S$-algebras}
\author{Andrew Baker \and Birgit Richter \and Markus Szymik }
\address{School of Mathematics \& Statistics, University
of Glasgow, Glasgow G12 8QW, Scotland.}
\email{a.baker@maths.gla.ac.uk}
\urladdr{\href{http://www.maths.gla.ac.uk/~ajb}
                            {http://www.maths.gla.ac.uk/~ajb}}
\address{Fachbereich Mathematik der Universit\"at Hamburg,
Bundesstrasse 55, 20146 Hamburg, Germany.}
\email{richter@math.uni-hamburg.de}
\urladdr{\href{http://www.math.uni-hamburg.de/home/richter/}
                    {http://www.math.uni-hamburg.de/home/richter/}}
\address{Mathematisches Institut, Heinrich-Heine-Universit\"at D\"usseldorf,
Universit\"atsstrasse 1, 40225 D\"usseldorf, Germany.}
\email{szymik@math.uni-duesseldorf.de}
\urladdr{\href{http://www.math.uni-duesseldorf.de/Personen/indiv/Szymik}
              {http://www.math.uni-duesseldorf.de/Personen/indiv/Szymik}}
\thanks{The first named author thanks the Fachbereich Mathematik
der Universit\"at Hamburg for its hospitality. The second named
author was supported by the Glasgow Mathematical Journal Learning
and Research Support Fund and thanks the Mathematics Department
of the University of Glasgow for its hospitality.}
\date{\version}
\begin{abstract}
We investigate a notion of Azumaya algebras in the context of
structured ring spectra and give a definition of Brauer groups.
We investigate their Galois theoretic properties, and discuss
examples of Azumaya algebras arising from Galois descent and
cyclic algebras. We construct examples that are related to
topological Hochschild cohomology of group ring spectra and
we present a $K(n)$-local variant of the notion of Brauer
groups.
\end{abstract}

\maketitle

\begin{center}
\textit{To the memory of Gor\^o Azumaya, $\ast$ 26th of February
  1920, $\dagger$ 8th of July 2010}
\end{center}

\section*{Introduction}

The investigation of Brauer groups of commutative $S$-algebras
is one aspect of the attempt to understand arithmetic properties
of structured ring spectra.

In classical algebraic settings, Brauer groups are defined in
terms of Azumaya algebras over fields or more generally over
commutative rings~\cite{Azu:first,MA&OG:BrauerGps,Saltman} and
are closely involved in Galois theoretic considerations. In
this paper we discuss some ideas on Brauer groups for commutative
$S$-algebras and in Section~\ref{sec:Gal->Azumaya} we investigate
their behaviour with respect to Galois extensions of commutative
$S$-algebras in the sense of John Rognes~\cite{JR:Opusmagnus}.
In earlier work, the first named author and Andrey Lazarev discussed
notions of Azumaya algebras \cite[sections 2, 4]{AB&AL:Morita}, but
these appear to be technically problematic: there faithfulness of
the underlying module spectra was not required, but many of standard
constructions with Azumaya algebras rely on this property. Also, the
link with other definitions (for instance \cite{To,NJ}) works only
under this faithfulness assumption.

Since our work was begun, other people have carried out work
on Azumaya algebras and Brauer groups in contexts related to
ours. Niles Johnson~\cite{NJ} discusses Azumaya objects for
general closed autonomous symmetric monoidal bicategories, 
and proves a comparison result~\cite[proposition~5.4]{NJ} 
which compares our definition of Azumaya algebras with his.
In~\cite[theorem~1.5]{NJ} he also shows that the derived 
Brauer group of a commutative ring in the sense of Bertrand 
To\"en agrees with our Brauer group of the corresponding 
Eilenberg-Mac~Lane ring spectrum.

In \cite[definition~4.6]{szymik} the third named author extends
our current work to construct a Brauer space for commutative
$S$-algebras such that the fundamental group of that space 
agrees with our Brauer group. The approaches in \cite{To,AG,GL} 
give descriptions of Brauer groups in terms of \'etale cohomology
groups in the derived context and the context of ring spectra,
respectively.

We present our  definition of topological Azumaya algebras 
in Section~\ref{sec:AzumayaAlg} and show that such algebras 
are always homotopically central (in the sense of 
Definition~\ref{def:htpcentral}) and separable, and also that 
the Azumaya property is preserved under base change.

In Section~\ref{sec:Brauer} we define Brauer groups of 
commutative $S$-algebras and in Section~\ref{sec:Gal->Azumaya} 
we prove a version of Galois descent for topological Azumaya 
algebras. This is applied in Section~\ref{sec:cyclic} where we 
explain how the classical theory of cyclic algebras can be 
extended to the context of commutative $S$-algebras.

In the case of Eilenberg-Mac~Lane spectra we show in
Section~\ref{sec:EM} under some assumptions on the ring $R$
that an extension $HR \lra HA$ is topologically Azumaya if
and only if the extension of commutative rings $R \lra A$ is
an algebraic Azumaya extension. Furthermore, using recent work
of Bertrand To\"en~\cite{To}, we can deduce that the Brauer
group $\Br(H\k)$ is trivial if $\k$ is an algebraically or
separably closed field.

Classically, the center of an associative algebra $A$ over a 
commutative ring $R$ can be described as endomorphisms of $A$ 
in the category of modules over the enveloping algebra 
$A^e = A \otimes_R A^\op$. For structured ring spectra, the direct 
analogue of this definition does not yield a homotopy invariant notion. 
Instead one has to replace $A$ by a cofibrant object in the category 
of module spectra over the enveloping algebra spectrum, so the center 
of an associative $R$-algebra spectrum $A$ is given by the topological 
Hochschild cohomology spectrum $\THH_R(A,A)$. This spectrum is not 
strictly commutative in general, but due to the affirmatively solved 
Deligne conjecture~\cite{MS} it is an $E_2$-spectrum. There are however 
exceptions and in Section~\ref{sec:grouprings} we discuss some examples 
arising from group ring spectra and their homotopy fixed point spectra.

In Section~\ref{sec:Azumaya-LT} we offer a variant of the construction
of Brauer groups in the $K(n)$-local context where it appears that
technical difficulties are minimized and we discuss some examples
related to $EO_2=L_{K(2)}\mathit{TMF}$ in Section~\ref{sec:K-localAzumaya}. 
In Section~\ref{sec:Br(SK)} we describe a non-trivial element in the 
$K(n)$-local Brauer group of the $K(n)$-local sphere.

\section{Azumaya algebras over commutative $S$-algebras}
\label{sec:AzumayaAlg}

Throughout, let $R$ be a cofibrant commutative $S$-algebra. We work
in the categories of $R$-modules, $\mathscr{M}_R$, and associative
$R$-algebras, $\mathscr{A}_R$ and for definiteness we choose the
framework of~\cite{EKMM}. Following~\cite{AB&BR:Galois,JR:Opusmagnus},
we will say that an $R$-module $W$ is \emph{faithful} if for an
$R$-module $X$, $W\wedge_R X\simeq *$ implies that $X\simeq *$.

We recall some ideas from~\cite{AB&AL:Morita}. If $A$ is
an $R$-algebra, we denote by $A^\op$ the $R$ algebra whose
underlying $R$-module is $A$ but whose multiplication is
reversed. The topological Hochschild cohomology spectrum
of $A$ (over $R$) is
\[
\THH_R(A)=\THH_R(A,A)=F_{A\wedge_R A^\op}(\tilde A,\tilde A),
\]
where $\tilde A$ is a cofibrant replacement for $A$ in the category
of left $A\wedge_R A^\op$-modules $\mathscr{M}_{A\wedge_R A^\op}$.
We write $\eta\:R\lra\THH_R(A)$ for the canonical map into the
$R$-algebra $\THH_R(A)$; we also write
$\mu\:A\wedge_R A^\op\lra F_R(A,A)$ for the $R$-algebra map induced
by the left and right actions of $A$ and $A^\op$ on $A$.

\begin{defn}\label{defn:Azumaya}
Let $A$ be an $R$-algebra. Then $A$ is a \emph{weak \emph{(}topological\emph{)}
Azumaya algebra over $R$} if and only if the first two of the following
conditions hold, while $A$ is a \emph{\emph{(}topological\emph{)}
Azumaya algebra over $R$} if and only if all three of them hold.
\begin{enumerate}
\item
$A$ is a dualizable $R$-module.
\item
$\mu\:A\wedge_R A^\op\lra F_R(A,A)$ is a weak equivalence.
\item
$A$ is faithful as an $R$-module.
\end{enumerate}
\end{defn}
Note that this definition of Azumaya algebras over~$R$ differs
from that in~\cite{AB&AL:Morita} since we demand faithfulness
of $A$ over $R$ and not just $A$-locality of $R$ as an $R$-module.

If $T$ is an ordinary commutative ring with unit and if $B$ is an
associative $T$-algebra, then the center of $B$ can be identified
with the endomorphisms of $B$ as an $B  \otimes_T B^\op$-module.
Therefore $\THH_R(A)$ can be viewed as a homotopy invariant version
of the center of $A$.
\begin{defn} \label{def:htpcentral}
An $R$-algebra $A$ is said to be \emph{homotopically central} if
the canonical map $\eta\:R\lra\THH_R(A)$ is a weak equivalence.
\end{defn}
For the following we recall a special case of the Morita theory
developed in \cite[section~1]{AB&AL:Morita}. For a topological
Azumaya algebra $A$ over $R$ we consider the category of left
modules over the endomorphism spectrum $F_R(A,A)$,
$\mathscr{M}_{F_R(A,A)}$ and we take a cofibrant replacement
$\bar{A}$ of $A$ in this category. The functor
\begin{equation*}
F\: \mathscr{M}_R \lra \mathscr{M}_{F_R(A,A)}
\end{equation*}
that sends $X$ to $X \wedge_R \bar{A}$ has an adjoint
\begin{equation*}
G \:\mathscr{M}_{F_R(A,A)} \lra \mathscr{M}_R
\end{equation*}
with $G(Y) =F_{F_R(A,A)}(\bar{A},Y)$. Then~\cite[theorem~1.2]{AB&AL:Morita}
implies that this adjoint pair of functors passes to an adjoint
pair of equivalences between the corresponding derived categories
\begin{equation*}
\xymatrix@1{
{\mathscr{D}_R}\ar@<0.5ex>[r]^{\widetilde{F}\ph{abc}}
      & \ar@<0.5ex>[l]^{G\ph{abc}} {\mathscr{D}_{F_R(A,A)}}
}
\end{equation*}
and as a direct consequence we obtain the following result.
\begin{prop}[{\cite[proposition~2.3]{AB&AL:Morita}}]\label{prop:htpcentral}
Every topological Azumaya algebra $A$ over $R$ is homotopically
central.
\end{prop}
By proposition~2.3 and definition~2.1 of~\cite{AB&AL:Morita}, we
also see that any topological Azumaya algebra $A$ over $R$ is
dualizable as an $A \wedge_R A^\op$-module and $A \wedge_R A^{\op}$
is $A$-local as a left module over itself.

In classical algebra, Azumaya algebras are in particular separable.
Using Morita theory we can deduce the analogous statement for
topological Azumaya algebras. Here an $R$-algebra is \emph{separable}
in the sense of~\cite[definition~9.1.1]{JR:Opusmagnus} if the
multiplication $m\: A \wedge_R A \lra A$ has a section in the
derived category of left $A \wedge_R A^\op$-modules,
$\mathscr{D}_{A \wedge_R A^\op}$.

\begin{prop} \label{prop:separable}
Let $A$ be a topological Azumaya $R$-algebra. Then $A$ is
separable.
\end{prop}
\begin{proof}
By the remark following~\cite[definition~9.1.1]{JR:Opusmagnus},
it suffices to prove that the induced map
\begin{equation*}
m_*\:\THH_R(A,A\wedge_R A) \lra \THH_R(A,A)
\end{equation*}
is surjective on $\pi_0(-)$. Denote by $\tilde{A}$ a cofibrant
replacement of $A$ in the category of $A \wedge_R A^\op$-modules.
Morita equivalence yields the two weak equivalences
\begin{align*}
\tilde{G}\circ \tilde{F}(R) &\simeq \THH_R(A,A), \\
\tilde{G}\circ \tilde{F}(A) &\simeq \THH_R(A,A\wedge_R A).
\end{align*}
The functoriality of $\tilde{G}\circ \tilde{F}$ ensures that
the unit $\eta\: R \lra A$ induces a map
${\tilde{G}\circ\tilde{F}(\eta)}$ with
\begin{equation*}
\xymatrix@1{ {R} \ar[r]^(0.35){\simeq} & {\tilde{G}\circ\tilde{F}(R)}
\ar[rr]^{\tilde{G}\circ
\tilde{F}(\eta)} &  & {\tilde{G}\circ\tilde{F}(A)} &
\ar[l]_(0.35){\simeq} {A}.
}
\end{equation*}
This is given by sending the coefficient module of $\THH$,
$\tilde{A} \simeq R\wedge_R\tilde{A} \simeq R \wedge_R A$,
to $A \wedge_R A \simeq A \wedge_R \tilde{A}$ using $\eta$.
Therefore
\begin{equation*}
\pi_0(m_*) \circ \pi_0(\tilde{G}\circ \tilde{F}(\eta)) = \id,
\end{equation*}
and so $\pi_0(m_*)$ is surjective.
\end{proof}

We now describe the behaviour of Azumaya algebras under base
change.
\begin{prop} \label{prop:basechange}
Let $A,B,C$ be $R$-algebras.
\begin{enumerate}
\item
If $A$ is an Azumaya algebra over $R$ and if\/ $C$ is a commutative
$R$-algebra, then $A \wedge_R C$ is an Azumaya algebra over\/ $C$.
\item
Conversely, let\/ $C$ be a commutative $R$-algebra such that\/ $C$
is dualizable and faithful as an $R$-module. If $A \wedge_R C$ is an
Azumaya algebra over\/ $C$, then $A$ is an Azumaya algebra over\/ $R$.
\item
If $A$ and $B$ are Azumaya algebras over $R$, then $A \wedge_R B$
is also Azumaya over $R$.
\end{enumerate}
\end{prop}
\begin{proof}
If $A$ is an Azumaya algebra over $R$, then it is formal
to verify that $A \wedge_R C$ is dualizable and faithful
over $C$ (compare~\cite[4.3.3, 6.2.3]{JR:Opusmagnus}). It
remains to show that
\begin{equation*}
\mu_{A \wedge_R C}\:(A \wedge_R C) \wedge_C (A \wedge_R C)^\op
                       \lra F_C(A \wedge_R C, A \wedge_R C)
\end{equation*}
is a weak equivalence. Note that since the multiplication
in $A\wedge_R C$ is defined componentwise,
\begin{equation*}
(A\wedge_R C)^\op = A^\op\wedge_R C^\op.
\end{equation*}
The diagram
\begin{equation} \label{eqn:mu}
\xymatrix{
{(A \wedge_R C) \wedge_C (A \wedge_R C)^\op } \ar[d]_{\simeq}
                        \ar[rr]^{\mu_{A \wedge_R C}} & {} &
{F_C(A \wedge_R C, A \wedge_R C)} \ar[d]^{\simeq}\\
{A \wedge_R A^\op \wedge_R C} \ar[dr]^{\mu_A \wedge_R C}& {}
                     & {F_R(A, A\wedge_R C)} \\
  & {F_R(A,A) \wedge_R C} \ar[ur]^{\nu} &
}
\end{equation}
commutes. Here $\nu\: F_R(A,A) \wedge_R C \lra F_R(A,A \wedge_R C)$
denotes the duality map. As $A$ is Azumaya over $R$ we know that
$\nu$ and $\mu_A$ are equivalences, and thus we obtain that the
top map is an equivalence as well.

For the converse we assume that $A \wedge_R C$ is Azumaya over
$C$ and $C$ is faithful and dualizable as an $R$-module. If $M$
is an $R$-module, then $A \wedge_R M \simeq *$ implies that
\begin{equation*}
(A \wedge_R C) \wedge_R M
      \simeq (A \wedge_R C) \wedge_C (C \wedge_R M) \simeq *.
\end{equation*}
Also, the faithfulness of $A \wedge_R C$ over $C$ ensures
that $C\wedge_R M \simeq *$. But as we assumed that $C$
is faithful over~$R$, we can conclude that $M$ was trivial.

The fact that $A$ is dualizable over $R$ follows
from~\cite[lemma~6.2.4]{JR:Opusmagnus}. Making use of
diagram~\eqref{eqn:mu} we see that $\mu_A$ is also a weak
equivalence.

The proof of the third claim is straightforward.
\end{proof}

Later we will consider Azumaya algebras in a Bousfield local
setting. Let $L$ be a cofibrant $R$-module.
\begin{defn}\label{defn:AzumAlg-Klocal}
An $L$-local $R$-algebra $A$ is an (\emph{$L$-local}) \emph{Azumaya
algebra} if
\begin{enumerate}
\item
$A$ is a dualizable $L$-local $R$-module.
\item
The natural morphism of $R$-algebras $A\wedge_R A^\op\lra F_R(A,A)$
is an $L$-local equivalence.
\item
$A$ is faithful as an $L$-local $R$-module.
\end{enumerate}
\end{defn}
Here dualizability as an $L$-local $R$-module means dualizability
in the derived category of $L$-local $R$-modules. This is a symmetric
monoidal category with the $L$-localization of the smash product 
over $R$ as the symmetric monoidal product, so the definition of 
dualizability from~\cite{AD&DP:Duality} applies.

\section{Brauer groups}\label{sec:Brauer}

Now suppose that $M$ is a dualizable $R$-module as discussed
in~\cite{JR:Opusmagnus,AB&BR:Galois}; a more detailed discussion
of dualizability can be found in~\cite{AD&DP:Duality}. Let
$\mathcal{E}_R(M)=F_R(M,M)$ be its endomorphism $R$-algebra. 
Then there is a weak equivalence
\begin{equation}\label{eqn:End=func}
\mathcal{E}_R(M) \simeq F_R(M,R) \wedge_R M.
\end{equation}
In order to identify endomorphism spectra of faithful and
dualizable $R$-modules as trivial Azumaya algebras we need
the following auxiliary result.

\begin{lem}\label{lem:Dual-faithful}
Let $M$ be a dualizable $R$-module.
\begin{enumerate}
\item
If $M$ is a faithful $R$-module, then the dual $F_R(M,R)$ is also
faithful.
\item
If $M$ is $L$-local with respect to a cofibrant $R$-module $L$,
then $F_R(M,R)$ is $L$-local.
\end{enumerate}
\end{lem}
\begin{proof}
(1) Dualizability of $M$ implies that the composition
\begin{equation*}
M \simeq R \wedge_R M 
   \xrightarrow{\delta\wedge\id} M \wedge_R F_R(M,R) \wedge_R M
   \xrightarrow{\id\wedge\epsilon} M \wedge_R R \simeq M
\end{equation*}
is the identity on $M$. Here $\delta\:R \lra M \wedge_R F_R(M,R)$
is the counit, and  $\epsilon \: F_R(M,R) \wedge_R M \lra R$ is 
the evaluation map. Now if $N$ is an $R$-module for which
$F_R(M,R) \wedge_R N \simeq *$, then the identity of $M \wedge_R N$
factors through the trivial map, hence  $N \simeq *$ by
faithfulness of $M$. \\
(2) A similar argument with the functor $F_R(W,-)$ shows that if
$L\wedge_R W\simeq *$, then the identity map on $F_R(W,F_R(M,R))$
factors through
\[
F_R(W,F_R(M,R)\wedge_R M \wedge_RF_R(M,R)) \simeq
    F_R(W \wedge_R M \wedge_R M, M) \simeq *.
\qedhere
\]
\end{proof}

It was shown in~\cite[proposition~2.11]{AB&AL:Morita} that if $M$
is a dualizable, cofibrant $R$-module, then $\mathcal{E}_R(M)$
is a weak topological Azumaya algebra in the sense
of~\cite[definition~2.1]{AB&AL:Morita}.
\begin{prop}\label{prop:End-Az}
If $M$ is a faithful, dualizable, cofibrant $R$-module, then
(a cofibrant replacement of) $\mathcal{E}_R(M)$ is an Azumaya
$R$-algebra.
\end{prop}
\begin{proof}
As $\mathcal{E}_R(M)$ is a weak Azumaya algebra, it suffices
to show that $\mathcal{E}_R(M)$ is a faithful $R$-module.
Dualizability of $M$ ensures that
\begin{equation*}
\mathcal{E}_R(M) \simeq F_R(M,R) \wedge_R M,
\end{equation*}
and this is a smash product of two faithful $R$-modules which
is also faithful.
\end{proof}

This result shows that we can take the $R$-algebras of the form
$\mathcal{E}_R(M)$ with $M$ faithful, dualizable and cofibrant,
to be trivial Azumaya algebras when defining a topological version
of a Brauer group which we now do.

First we note that every Azumaya algebra is weakly equivalent
to a retract of a cell $R$-module, so the following construction
yields a \emph{set} of equivalence classes. Define $\Az(R)$
to be the collection of all Azumaya algebras. Now we introduce
our version of the Brauer equivalence relation $\approx$ on
$\Az(R)$.
\begin{defn}\label{defn:Brauer-eqce}
Let $R$ be a cofibrant commutative $S$-algebra.
If $A_1,A_2\in\Az(R)$, then $A_1 \approx A_2$ if and only
if there are faithful, dualizable, cofibrant $R$-modules
$M_1,M_2$ for which
\[
A_1\wedge_R F_R(M_1,M_1)\simeq A_2\wedge_R F_R(M_2,M_2)
\]
as $R$-algebras. We denote the set of equivalence classes
of these by $\Br(R)$ and we use the notation $[A]$ for the
equivalence class of an $R$-Azumaya algebra $A$.
\end{defn}
\begin{thm}\label{thm:BrauerGps}
The set $\Br(R)$ is an  abelian group with multiplication
induced by the smash product $\wedge_R$. Furthermore, $\Br$
is a functor from the category of commutative $S$-algebras
to abelian groups.
\end{thm}
\begin{proof}
The details involve routine modifications of the approach
used in the case of Brauer groups of commutative rings
in~\cite[theorem~5.2]{MA&OG:BrauerGps}.

Functoriality for  morphisms of commutative $S$-algebras
$R\lra R'$ is achieved by sending an $R$-algebra $A$ to
the $R'$-algebra $R'\wedge_R A$.
\end{proof}

\begin{rem}\label{rem:NJ-EW}
Johnson's work~\cite[lemma~5.7]{NJ} shows that two
Azumaya $R$-algebras are Brauer equivalent in the sense of
Definition~\ref{defn:Brauer-eqce} if and only if they are
Eilenberg-Watts equivalent in the sense of~\cite[definition~1.1]{NJ}.
As a consequence any Azumaya algebra that is Brauer equivalent
to $R$ is weakly equivalent to $F_R(M,M)$ for some dualizable
faithful cofibrant $R$-module spectrum $M$.
\end{rem}

For a cofibrant $R$-module $L$, we can similarly define 
the sets of $L$-local Azumaya algebras $\Az_L(R)$ and 
the associated $L$-local Brauer group $\Br_L(R)$.

In order to relate Azumaya algebras to Galois theory, we
require the following notions modelled on algebraic analogues.
\begin{defn}\label{defn:Azumaya-split}
Let $R\lra R'$ be an extension of commutative $S$-algebras.
Then the Azumaya algebra $R\lra A$ is \emph{split} by
$R\lra R'$ (or just by $R'$) if $R'\wedge_R A\approx R'$,
or equivalently if $[A]\in\ker(\Br(R)\lra\Br(R'))$. We define
the \emph{relative Brauer group}
\[
\Br(R'/R) = \ker(\Br(R)\lra\Br(R')).
\]
Similarly we can define a \emph{relative $L$-local Brauer
group}
\[
\Br_L(R'/R) = \ker(\Br_L(R)\lra\Br_L(R')).
\]
\end{defn}

In practise, we will use this when $R\lra R'$ is a faithful
$G$-Galois extension for some finite group~$G$.

\section{Galois extensions and  Azumaya algebras}
\label{sec:Gal->Azumaya}

Consider a map of commutative $S$-algebras $A \lra B$, which
we often denote by $B/A$. If~$A$ is cofibrant as a commutative
$S$-algebra, $B$ is cofibrant as a commutative $A$-algebra,
and if $G$ is a finite group which acts on $B$ by morphisms of
commutative $A$-algebras, then following Rognes~\cite{JR:Opusmagnus},
then we call $B/A$ a \emph{$G$-Galois extension} if the canonical
maps $i\: A \lra B^{hG}$ and $h \: B \wedge_A B \lra F(G_+,B)$
are weak equivalences.

In addition to these conditions, we will assume that $B$ is
faithful as an $A$-module spectrum. This is a further restriction
as there are examples of Galois extensions which are not faithful.
The following example is due to Wieland (see \cite{JR:notff}).
\begin{rem}\label{rem:Wieland}
Let $p$ be a prime. Then the $\Z/p$-Galois extension
\begin{equation*}
F(B\Z/p_+,H\F_p) \lra F(E\Z/p_+,H\F_p) \simeq H\F_p
\end{equation*}
is not faithful. To its eyes the $\Z/p$-Tate spectrum of $H\F_p$
appears trivial, but it is not.
\end{rem}

Let $B\<G\>$ be the twisted group algebra over $B$, \ie, the
$A$-algebra whose underlying $A$-module is $B \wedge G_+$ and 
whose multiplication is the composition $\tilde{\mu}$
\begin{equation*}
\xymatrix{
{B \wedge G_+ \wedge B \wedge G_+} \ar[rr]^{\id \wedge \Delta\wedge\id\ph{abc}}
\ar@/_20pt/[ddrrrr]_{\tilde{\mu}} & &  {B \wedge G_+ \wedge G_+ \wedge
  B \wedge  G_+}
\ar[rr]^{\ph{abc}\id\wedge \nu \wedge \id} & & {B \wedge G_+  \wedge B \wedge G_+}
\ar[d]^{(2,3)} \\
& & & & {B \wedge B  \wedge G_+ \wedge G_+} \ar[d]^{\mu_B \wedge \mu_G} \\
& & & &  {B \wedge G_+}}
\end{equation*}
where $\Delta$ is the diagonal, $\nu$ denotes the $G$-action
on $B$, $\mu_B$ is the multiplication of $B$ and $\mu_G$ the
multiplication in $G$. Then $\tilde{\mu}$ factors through
$(B\wedge G_+)\wedge_A(B\wedge G_+)$ and turns $B\<G\>$ into
an $A$-algebra. Note that $B\<G\>$ is an associative algebra
but in general it lacks commutativity. More precisely, we
know that the morphism $j\: B\<G\> \lra F_A(B,B)$ is a weak
equivalence of $A$-algebras for every $G$-Galois extension
$A \lra B$. In particular, $B\<G\>$ gives rise to a trivial
element in the Brauer group of $A$.

\begin{lem}\label{lem:Azumaya-module}
Let $B/A$ be a faithful $G$-Galois extension and let $M$ 
be a $B\<G\>$-module which is of the form $B \wedge_A N$ 
for some $A$-module $N$, where the $B\<G\>$-module structure 
is given by the $B$-factor of $B \wedge_A N$. Then there 
is a weak equivalence of $A$-modules $N \simeq M^{hG}$.
\end{lem}
\begin{proof}
Consider $B \wedge_A M = B \wedge_A B \wedge_A N$. As $B$
is $G$-Galois over $A$, the latter term is equivalent to
$F(G_+,B)\wedge_A N$ and this in turn is equivalent to
$F(G_+,B \wedge_A N)$ because $G_+$ is finite. As $B$ is
dualizable over $A$, the homotopy fixed point spectrum
$(B\wedge_A M)^{hG}$ is equivalent to $B \wedge_A M^{hG}$.

There is a chain of equivalences of $B$-modules
\begin{equation*}
B \wedge_A N \xrightarrow{\;\simeq\;} F(G_+,B \wedge_A N)^{hG}
    \xrightarrow{\;\simeq\;} (B \wedge_A B \wedge_A N)^{hG}
                                     = (B \wedge_A M)^{hG}
    \xleftarrow{\;\simeq\;} B \wedge_A M^{hG},
\end{equation*}
and the result follows by faithfulness of $B$ over $A$.
\end{proof}

The following two results give analogues of Galois descent 
of algebraic Azumaya algebras as in~\cite[proposition~6.11]{Saltman}.
\begin{prop}\label{prop:Azumaya-fixedpts}
Suppose that\/ $C$ is an Azumaya algebra over\/ $B$ for 
which the natural morphism $B \wedge_A C^{hG}\lra C$ is 
a weak equivalence of $B\<G\>$-modules. Then $C^{hG}$ is 
also an Azumaya algebra over~$A$.
\end{prop}
\begin{proof}
We know from \cite[lemma~6.2.4]{JR:Opusmagnus} that the 
$A$-algebra $C^{hG}$ is dualizable as an $A$-module.

As $C$ is Azumaya over $B$, we know that $C\wedge_B C^\op\simeq F_B(C,C)$.
Also, dualizability of $C^{hG}$ over $A$ guarantees that
\begin{align*}
    B \wedge_A F_A(C^{hG}, C^{hG}) &\simeq F_A(C^{hG},B \wedge_A C^{hG}) \\
    &\iso F_B(B \wedge_A C^{hG}, B \wedge_A C^{hG}) \\
    &\simeq F_B(C,C) \simeq C\wedge_B C^\op,
\end{align*}
and so
\begin{align*}
C \wedge_B C^\op
  &\simeq (B \wedge_A C^{hG})\wedge_B (B \wedge_A (C^{hG})^\op) \\
  &\simeq B \wedge_A (C^{hG} \wedge_A (C^{hG})^\op).
\end{align*}
As $B$ is faithful over $A$, this shows that
\begin{equation*}
C^{hG} \wedge_A(C^{hG})^\op \simeq  F_A(C^{hG}, C^{hG}).
\end{equation*}

Since $C$ is faithful as a $B$-module and $B$ is faithful 
as an $A$-module, we know that $C$ is faithful as an 
$A$-module. Assume that for an $A$-module $M$ we have
$C^{hG}\wedge_A M\simeq*$. This is the case if and only if
\begin{equation*}
B \wedge_A C^{hG} \wedge_A M \simeq C \wedge_A M \simeq *
\end{equation*}
because $B$ is a faithful $A$-module. Now faithfulness of
$C$ over $A$ implies that $C^{hG}$ is also faithful over~$A$.
\end{proof}

Suppose that $B/A$ is a faithful $G$-Galois extension in 
the sense of Rognes~\cite{JR:Opusmagnus}, where~$G$ is a 
finite group. Now let $H\lhd K\leq G$ so that $B/B^{hH}$ 
is a faithful $H$-Galois extension, $K$ acts on $B^{hH}$ 
by $B^{hK}$-algebra maps and $B^{hK}\lra B^{hH}$ is a 
faithful $K/H$-Galois extension, in particular,
\begin{equation}\label{eqn:semi-dir}
B^{hK} \simeq (B^{hH})^{h(K/H)}.
\end{equation}

By~\cite[lemma~6.1.2(b)]{JR:Opusmagnus}, the twisted group 
ring $B\<H\>\simeq F_{B^{hH}}(B,B)$ is an Azumaya algebra 
over $B^{hH}$, and $K$ acts on $B\<H\>$ by extending the 
action on $B$ by conjugation on~$H$, so we will write 
$B\<H_c\>$ to emphasize this.

If $K=Q\ltimes H$ is a semi-direct product or $H$ is abelian,
the quotient $Q=K/H$ acts by conjugation on~$H$.

Note that as in algebra, there is an isomorphism of $A[K]$-modules
\begin{equation*}
A[K] \iso \prod_K A.
\end{equation*}
The algebraic version of this isomorphism is given by
\begin{equation*}
\sum_{k \in K} a_k k \leftrightarrow (a_{k^{-1}})_{k \in K}
\end{equation*}
and we will use the topological analogue of this.

Our next result is based on~\cite[proposition~6.11(b)]{Saltman}.
\begin{prop}\label{prop:Saltman-6.11}
Suppose that $K=Q\ltimes H$ is a semi-direct product, or that
$H$ is abelian. Then the $B^{hK}$-algebra $B\<H_c\>^{hQ}$ is
Azumaya, and
\begin{equation*}
B^{hH}\wedge_{B^{hK}}B\<H_c\>^{hQ} \simeq B\<H_c\>.
\end{equation*}
Hence the Azumaya algebra $B\<H_c\>^{hQ}$ over $B^{hK}$ 
is split by~$B^{hH}$.
\end{prop}
\begin{proof}
Note that we can assume that $G=K$ and $B^{hK}=A$. Making 
use of a faithful base change, it suffices to assume that 
$B$ is the trivial $K$-Galois extension, $B=\prod_K A$.

There are isomorphisms of $A[K]$-modules
\begin{align}
B\<H_c\> &\iso {}_{\mathrm{diag}}(\prod_K A \wedge_{A} A[H_c]) \notag\\
       &\iso {}_{\mathrm{left}}(\prod_K A \wedge_A A[H]) \notag\\
       &\iso {}_{\mathrm{left}}(A[K] \wedge_A A[H]) \notag\\
       &\iso {}_{\mathrm{left}}(A[K\times H]), \label{eqn:eqt-iso}
\end{align}
where ${}_{\mathrm{diag}}(-)$ and ${}_{\mathrm{left}}(-)$ indicate 
the diagonal and left $K$-actions respectively, the second isomorphism 
is the standard equivariant shear map similar to the map $\mathrm{sh}$
of~\cite[section~3.5]{JR:Opusmagnus}, and $K\times H$ is viewed as 
a $K$-set through the action on the left hand factor. As a $Q$-set, 
$K$ decomposes into free orbits indexed on $H$. On taking $Q$-homotopy
fixed points we obtain an equivalence of $A$-modules
\begin{equation}\label{eqn:eqt-iso-hfixed}
B\<H_c\>^{hQ} \iso A[H\times H].
\end{equation}
There is a map of $A$-modules
\begin{equation*}
B^{hH} \xrightarrow{\;\mathrm{unit}\;}
       B^{hH}\wedge_A B\<H_c\>^{hQ} \lra B\<H_c\>
\end{equation*}
which is also a map of $B^{hH}\<Q\>$-modules. Applying $\pi_*(-)$
and working algebraically with $\pi_*(A)$-modules,
using~\cite[proposition~6.11(b)]{Saltman} it follows that we have
an isomorphism
\begin{equation*}
\pi_*(B^{hH}\wedge_A B\<H_c\>^{hQ}) \iso \pi_*(B\<H_c\>),
\end{equation*}
and therefore a weak equivalence
\begin{equation*}
B^{hH}\wedge_A B\<H_c\>^{hQ} \xrightarrow{\;\simeq\;} B\<H_c\>
\end{equation*}
of $B^{hH}\<Q\>$-modules. Now Proposition~\ref{prop:Azumaya-fixedpts}
shows that $B\<H_c\>^{hQ}$ is Azumaya over $B^{hK}$.
\end{proof}

\section{Cyclic algebras}
\label{sec:cyclic}

In this section, we will assume that~$K\to L$ is a faithful
Galois extension of commutative~$S$-algebras, and that the
Galois group~$G=\Gal(L/K)$ is generated by an element~$\sigma$
of order~$n$, say. In particular, this group has to be cyclic.
The choice of a generator corresponds to an isomorphism~$\Z/n\cong G$,
whose inverse can be thought of as a (primitive) character of~$G$.
Last but not least, we also need a strict unit~$u$ in~$K$. For
the time being, this just means that there is an action of the
group~$\Z$ of integers on the spectrum~$K$ via maps of~$K$-modules.
We will extend this action to~$L$ without change of notation. The
strictness of $u$ is needed in order to ensure that relations hold
on the nose and this in turn is necessary to obtain a strictly
associative algebra extension.

Cyclic~$K$-algebras will be defined here via Galois descent
from matrix algebras over~$L$. As a model for the matrix
algebra we use
\begin{equation*}
	M_n(L)=\bigvee_{i,j=1}^nL_{i,j},
\end{equation*}
with all~$L_{i,j}=L$ and multiplication given on summands
\begin{equation*}
	L_{i,j}\wedge_K L_{j,k}\lra L_{i,k}
\end{equation*}
by the multiplication in~$L$. One could also work with the
endomorphism $K$-algebra spectrum $F_K(\bigvee_{n}L,\bigvee_{n}L)$,
but this mixes covariant and contravariant behaviour in $\bigvee_n L$
and that is inconvenient for the explicit formulae that we need.

The cyclic group~$\Z/n$ acts on the~$L$-algebra~$M_n(L)$
component-wise, \ie, the generator acts as~$\sigma\: L_{i,j}\lra L_{i,j}$
on each summand. The multiplication and unit~$L\lra M_n(L)$
are equivariant, and we have equivalences
\begin{equation*}
	M_n(L)^{h\Z/n}\simeq M_n(L^{h\Z/n})\simeq M_n(K)
\end{equation*}
of~$K$-algebras. Something possibly more interesting happens
when we twist this action with the chosen unit~$u$: we define
a self-map on~$M_n(L)$ as the composition
\begin{equation}\label{eq:twistedaction}
\xymatrix@1@C=50pt{
	L_{i,j}\ar[r]^-{\id}&
	L_{i+1,j+1}\ar[r]^-{u^{\delta_{i,n}-\delta_{j,n}}}&
	L_{i+1,j+1}\ar[r]^-{\sigma}&
	L_{i+1,j+1},
	}
\end{equation}
where the indices~$i+1$ and~$j+1$ are read modulo~$n$,
and~$\delta$ is the Kronecker symbol.

\begin{lem} \label{lem:multiplicativity}
The above self-map generates a~$\Z/n$-action on~$M_n(L)$
as an associative~$K$-algebra.
\end{lem}

\begin{proof}
This follows as in algebra using the fact that
\begin{equation*}
u\wedge_{K}L=L\wedge_K u\:
	L \wedge_{K}L \lra L \wedge_{K} L
\end{equation*}
since~$u$ is a unit in~$K$. This guarantees that the~$K$-algebra
multiplication on~$L$ behaves well with respect to the twisted
action. Together with the naturality and symmetry of the fold
map this proves the claim.
\end{proof}

\begin{defn}\label{defn:CyclicAlg}
The \emph{cyclic~$K$-algebra}
\begin{equation*}
		A(L,\sigma,u)=M_n(L)^{h\Z/n},
\end{equation*}
associated with~$L$,~$\sigma$, and~$u$ is obtained
from the matrix algebra~$M_n(L)$ with the
twisted~$\Z/n$-action by passage to homotopy fixed
points.
\end{defn}

The following result shows that the cyclic~$K$-algebra
$A(L,\sigma,u)$ defines a class in the relative Brauer
subgroup~$\Br(L/K)$ of~$\Br(K)$.

\begin{thm}\label{thm:CyclicAlg-splits}
The cyclic algebra~$A(L,\sigma,u)$ is an Azumaya algebra
over~$K$ which splits over~$L$.
\end{thm}

\begin{proof}
Proposition~\ref{prop:Azumaya-fixedpts} above says that
if we have an Azumaya algebra~$B$ over~$L$ with a
compatible~$G$-action such that the natural morphism
\begin{equation*}
L\wedge_KB^{hG}\lra B
\end{equation*}
is an equivalence of~$L \< G\>$-modules, then~$A=B^{hG}$
is also an Azumaya algebra over~$K$. We want to apply
this here to the situation~$B=M_n(L)$ and~$G\cong\Z/n$.

Since~$L$ is a dualizable~$K$-module,
\begin{equation*}
L\wedge_{K}M_n(L)^{hG}
                    \simeq (L\wedge_{K}M_n(L))^{hG},
\end{equation*}
where~$G$ acts only on the right hand factor in
$L\wedge_{K}M_n(L)$.

As~$L$ is a~$G$-Galois extension of~$K$, we have
$L\wedge_{K}L\simeq \Map(G_+,L)$, and therefore
\begin{equation*}
L\wedge_{K}M_n(L) \cong M_n(L\wedge_{K}L)
                      \simeq M_n(\Map(G_+,L))
                      \simeq \Map(G_+,M_n(L)).
\end{equation*}
As the latter is equivariantly equivalent to~$L[G]$,
we see that
\begin{equation*}
	L\wedge_{K}M_n(L)^{hG} \simeq M_n(L)
\end{equation*}
and this yields the result.
\end{proof}

Of course, it may happen that a cyclic~$K$-algebra
$A(L,\sigma,u)$ represents the trivial element in
the Brauer group~$\Br(K)$ of~$K$. This depends very
much on the chosen unit~$u$, for example. One way
to prove non-triviality is to compute the homotopy
groups of~$A(L,\sigma,u)$, and to compare the result
with the homotopy groups of the representatives of
the elementary Azumaya algebras.

For this and other reasons, it is useful to know that
one may calculate the homotopy groups of the cyclic
algebra~$A(L,\sigma,u)$ by means of the homotopy fixed
point spectral sequence
\begin{equation}\label{eqn:KU-Quat-MS-LixPtSS}
\mathrm{E}_2^{s,t} = \mathrm{H}^s(\Z/n,\pi_tM_n(L))
                    \Lra\pi_{t-s}(M_n(L)^{h\Z/n}).
\end{equation}

The action on~$\pi_tM_n(L)\cong M_n(\pi_tL)$ is generated
by whatever the twisted action~\eqref{eq:twistedaction}
induces in homotopy:
\begin{equation*}
a_{i,j}\longmapsto
   \sigma_*u_*^{\delta_{i,n}-\delta_{j,n}}(a_{i+1,j+1}).
\end{equation*}
But, if we identify~$(\pi_tL)^{\oplus n}$ with the first
row (or column) of~$\pi_tM_n(L)$, we easily find
\begin{equation*}
	\pi_tM_n(L)\cong\Z[\Z/n]\otimes_\Z(\pi_tL)^{\oplus n},
\end{equation*}
which shows that~$\pi_tM_n(L)$ is an induced~$\Z/n$-module.
Therefore the~$E_2$-term of the homotopy fixed point spectral
sequence~\eqref{eqn:KU-Quat-MS-LixPtSS} vanishes above
the $0$-line, which shows
\begin{equation*}
	\pi_t A(L,\sigma,u)\cong(\pi_tL)^{\oplus n},
\end{equation*}
additively. In fact, this determines the underlying homotopy
type.

\begin{thm}\label{thm:Order2}
If~$n=2$, and if the unit $u$ has order $2$ in the sense that
it comes from a~$\Z/2$-action, then the class of~$A(L,\sigma,u)$
has order at most~$2$ in the Brauer group of~$\Br(K)$.
\end{thm}

\begin{proof}
We prove that there is a~$\Z/2$-equivariant equivalence of
$K$-algebras between~$M_2(L)$ and~$M_2(L)^\op$. As in algebra,
this equivalence is given by the transposition of matrices,
which we have to model by the permutation~$L_{i,j}\lra L_{j,i}$.
Then the same proof as in algebra shows that this is a map of
associative~$K$-algebras. In order to show that the action passes
to the homotopy fixed points we have to prove that it is compatible
with the twisted action of~\eqref{eq:twistedaction} that we impose
on~$M_2(L)$. But this is trivial except for the~$u$-action, where
we have to use that $u=u^{-1}$ is a second root of unity.
\end{proof}

\section{Azumaya algebras over Eilenberg-Mac~Lane spectra}
\label{sec:EM}

In this section we consider the case of Azumaya algebras over
the Eilenberg-Mac~Lane spectrum of a commutative ring. In~\cite{To},
To\"en introduces the algebraic notion of a \emph{derived Azumaya
algebra} over a commutative ring as a special case of the more
general notion for simplicial rings. First we explain how the
topological and algebraic notions are related.

In~\cite[section~IV.2]{EKMM}, an equivalence of categories
\begin{equation}\label{eqn:HR->R}
\Psi\:\mathscr{D}_{HR} \lra \mathscr{D}_R
\end{equation}
is constructed, where $\Psi$ is defined on a CW $HR$-module
$M$ to be the cellular chain complex $C_*(M)$.
By~\cite[proposition~IV.2.5]{EKMM}, for CW $HR$-modules $M,N$
there are isomorphisms of chain complexes of $R$-modules
\begin{align*}
C_*(M\wedge_{HR}N) &\iso C_*(M)\otimes_R C_*(N), \\
C_*(F_{HR}(M,N)) &\iso \Hom_R(C_*(M),C_*(N)).
\end{align*}
The inverse functor $\Phi=\Psi^{-1}$ also preserves the
monoidal structure, thus we have an equivalence of symmetric
monoidal categories.

Following To\"en~\cite{To}, see remark~1.2, and \cite[theorem~1.5]{NJ}
we find that an Azumaya algebra $A$ over $HR$, corresponds to
a derived Azumaya algebra over $R$. Note that as we are working
with associative  (but not commutative) $HR$-algebras, a cofibrant
$HR$-algebra is a retract of a cell $HR$-module relative to~$HR$
by~\cite[theorem~VII.6.2]{EKMM}.

We get the following correspondence whose assumptions are satisfied
when $R$ is a principal ideal domain for instance.
\begin{prop}\label{prop:EM}
Let $R$ be a commutative ring such that for any finitely presented
$R$-module $M$ with $\Tor_k^R(M,M) = 0$ for $k>0$ we can deduce 
that $M$ is flat over $R$.

Let $T$ be an $R$-algebra. Then the $HR$-algebra $HT$ is a topological
Azumaya algebra if and only if\/ $T$ is an algebraic Azumaya $R$-algebra.
\end{prop}
\begin{proof}
One direction is easy to see: if $R \lra T$ is an algebraic Azumaya
extension, then $HR \lra HT$ is topologically Azumaya without any
additional assumptions on $R$.

For the converse, from~\cite[theorem~IV.2.1]{EKMM} we have
\begin{align}
\label{eqn:SS-Tor-*}
\pi_n(HT\wedge_{HR}HT^\op) &= \Tor^R_n(T,T^\op), \\
\label{eqn:SS-Ext-*}
\pi_n(F_{HR}(HT,HT)) &= \Ext_R^{-n}(T,T).
\end{align}
Because $\Tor^R_s=0=\Ext_R^s$ when $s<0$, the Azumaya condition
\begin{equation*}
\mu\: HT \wedge_{HR} HT^\op \xrightarrow{\;\simeq\;} F_{HR}(HT,HT)
\end{equation*}
implies that for $n\neq0$,
\begin{equation}\label{eqn:Tor-Ext-0}
\pi_n(HT\wedge_{HR}HT^\op) = \Tor^R_n(T,T^\op) = 0
                  = \Ext_R^n(T,T) = \pi_n(F_{HR}(HT,HT)).
\end{equation}
In particular,
\begin{equation}\label{eqn:HR-Azumaya}
T\otimes_R T^\op = \pi_0(HT\wedge_{HR}HT^\op)
                 \iso \pi_0(F_{HR}(HT,HT)) = \Hom_R(T,T).
\end{equation}
According to \cite[remark~1.2]{To}, the $R$-module $T$ is finitely
presented and flat by assumption, therefore it is finitely generated
and projective by the corollary to~\cite[theorem~7.12]{Matsumura}.

For faithfulness, suppose that $M$ is a non-trivial $R$-module.
Since $HT$ is a faithful $HR$-module, $HT\wedge_{HR}HM \not\simeq*$.
Flatness of $T$ over $R$ together with~\cite[theorem~IV.2.1]{EKMM}
yields the isomorphisms
\begin{equation*}
\pi_*(HT\wedge_{HR}HM) \iso \pi_0(HT\wedge_{HR}HM) \iso T \otimes_R M,
\end{equation*}
and therefore $T \otimes_R M$ is not trivial.
\end{proof}
\begin{prop} \label{prop:real0}
For any commutative ring with unit $R$ there is a natural homomorphism
\begin{equation*}
H\:\Br(R) \lra \Br(HR)
\end{equation*}
induced by the functor which sends a ring to its Eilenberg-Mac~Lane
spectrum.
\end{prop}
\begin{proof}
Let $[A]$ be an element of $\Br(R)$, then Proposition~\ref{prop:EM}
identifies $HA$ as an $HR$-Azumaya algebra. If $[A]=0$, \ie, if
there is a finitely generated faithful projective $R$-module $M$
with $A \iso \Hom_R(M,M)$, then
\begin{equation*}
HA \simeq H\Hom_R(M,M) \simeq F_{HR}(HM,HM)
\end{equation*}
and therefore $HA$ is trivial in $\Br(HR)$.
\end{proof}
\begin{rem}\label{rem:inj}
To\"en shows that the Brauer group of derived Azumaya algebras over
$R$ is parametrized by
$H^2_{\text{\'et}}(R,\mathbb{G}_m) \times H^1_{\text{\'et}}(R,\Z)$
whereas the ordinary Brauer group of $R$, $\Br(R)$, corresponds to
the torsion part in $H^2_{\text{\'et}}(R,\mathbb{G}_m)$. Combining
this with the comparison result of Johnson~\cite[theorem~1.5]{NJ}
yields that the above homomorphism $H$ is injective for any $R$
corresponding to
$H^2_{\text{\'et}}(R,\mathbb{G}_m)_{\mathrm{tor}}
                \subseteq H^2_{\text{\'et}}(R,\mathbb{G}_m)$.
\end{rem}

The situation is drastically different if we consider arbitrary
$HR$-algebra spectra $A$. For instance, for every $R$, every
$R$-module spectrum $\Sigma^nHR$ is faithful and dualizable,
and therefore $F_{HR}(HR \vee \Sigma^nHR, HR \vee \Sigma^nHR)$
is a trivial topological Azumaya $HR$-algebra whose homotopy
groups spread over positive and negative degrees. This indicates
that the Eilenberg-Mac~Lane functor of Proposition~\ref{prop:real0}
will not induce an isomorphism in general. The relationship with
\'etale cohomology groups in~\cite{To} in fact shows that there
are derived Azumaya algebras that are not Brauer equivalent to
an ordinary Azumaya algebra. To\"en describes a concrete example
in~\cite[section~4]{To} originating in an example by Mumford.

We will discuss the case of $H\k$ for a field $\k$. If $A$ is
Azumaya over $H\k$, then as $A$ is dualizable over $H\k$ we
know that the homotopy groups of $A$ are concentrated in finitely
many degrees, say $\pi_r(A)\neq0$ only when $-m \leq r \leq n$
for some $m,n\geq0$. As $\k$ is a field, we have
\begin{equation*}
\pi_*(A \wedge_{H\k} A^\op) \cong \pi_*(A) \otimes_\k\pi_*(A)^\op.
\end{equation*}
Using the fact that $\mu$ induces an isomorphism, we can deduce
that $n = m$ because otherwise the kernel of $\pi_*(\mu)$ would
be nontrivial.

A derived Azumaya algebra over the field $\k$ is a differential
graded $\k$-algebra $B_*$ whose underlying chain complex is a
compact generator of the derived category of chain complexes of
$\k$-vector spaces $\mathscr{D}_\k$ such that the natural map
\begin{equation*}
\mu_{B_*} \: B_* \otimes_\k B_*^{\op} \lra \Hom_\k(B_*,B_*)
\end{equation*}
is an isomorphism in $\mathscr{D}_\k$. Here $B_* \otimes_\k B_*^{\op}$
agrees with the derived tensor product because we are working
over a field, and similarly, $\Hom_\k(B_*,B_*)$ is the graded
$\k$-vector space of derived endomorphisms of $B_*$. Now we can
relate topological $H\k$-Azumaya algebras to derived Azumaya
algebras over $\k$.

\begin{prop} \label{prop:derazk}
If $A$ is a topological Azumaya algebra over $H\k$, then 
$\pi_*(A)$ is a derived Azumaya algebra over $\k$.
\end{prop}
\begin{proof}
As $A$ is dualizable over $H\k$, its homotopy groups build 
a finite dimensional graded $\k$-vector space and hence 
$\pi_*(A)$ is a compact generator of $\mathscr{D}_\k$. The 
weak equivalence
\begin{equation*}
\mu\: A \wedge_{H\k} A^\op \lra F_{H\k}(A,A)
\end{equation*}
yields isomorphisms
\begin{equation*}
\mu_{\pi_*(A)} \: \pi_*(A) \otimes_\k \pi_*(A)^\op
                       \cong \pi_*(A\wedge_{H\k} A^\op)
                       \cong \pi_*F_{H\k}(A,A)
                       \cong \Hom_\k(A_*,A_*)
\end{equation*}
and so $\pi_*(A)$ is a derived Azumaya algebra over $\k$.
\end{proof}
Using Proposition~\ref{prop:derazk} together with To\"en's
results of~\cite[section~1]{To} we obtain the following.
\begin{thm}\label{thm:Br-Hk}
For any algebraically closed field $\k$, the Brauer group
of $H\k$ is trivial.
\end{thm}
\begin{proof}
Let $A$ be a derived Azumaya algebra over $\k$. We know
from~\cite[corollary~1.11]{To} that every derived Azumaya
algebra over an algebraically closed field $\k$, in particular
$\pi_*(A)$, is quasi-isomorphic to a graded $\k$-vector space
$\Hom_\k(V,V)$ for some finite dimensional graded $\k$-vector
space $V$.

Let
\begin{equation*}
M = HV = \bigvee_{i=1}^n\Sigma^{m_i} H\k
\end{equation*}
be the $H\k$-module spectrum such that $\pi_*M \cong V$ as
graded $\k$-vector spaces. Then $A$ is weakly equivalent to
$F_{H\k}(A,A)$ since there are isomorphisms
\begin{equation*}
\pi_*(A) \cong \Hom_\k(V,V) \cong \pi_*(F_{H\k}(M,M)).
\end{equation*}
Therefore $[A]$ is trivial in the Brauer group $\Br(H\k)$.
\end{proof}

\begin{rem}
Using \cite[corollary~1.15]{To} one can extend the result
to obtain the triviality of the Brauer group $\Br(H\k)$
for any separably closed field $\k$.
\end{rem}

\section{Realizability of algebraic Azumaya extensions}\label{sec:real}

Using Angeltveit's obstruction theory~\cite[theorem~3.5]{VA:THH},
we can import algebraic Azumaya algebra extensions into topology. 
Let $R$ be a commutative $S$-algebra and let $\pi_0R \lra A_0$ be 
an algebraic Azumaya extension. Then
\begin{equation*}
A_* := \pi_*R \otimes_{\pi_0 R} A_0
\end{equation*}
is a projective module over $R_* = \pi_*R $ and there is an $R$-module
spectrum $A'$ with $\pi_*(A') \cong A_*$ which can be built as a mapping
telescope of an idempotent corresponding to viewing $A_*$ as a direct
summand of a free $R_*$-module. The methods of~\cite{AB&BR:Galois} carry
over to give a homotopy associative $R$-ring spectrum $A$ that realizes
$A_*$ as the homotopy ring $\pi_*A$.

Angeltveit's obstruction theory \cite{VA:THH} then yields the 
following.
\begin{thm} \label{thm:real}
There is a unique $A_\infty$ $R$-algebra structure on $A$, \ie, 
there is a unique rigidification $r(A)$ of $A$ to an associative 
$R$-algebra. The resulting extension $R \lra r(A)$ is an Azumaya 
algebra.
\end{thm}
\begin{proof}
The existence of the $A_\infty$ structure on $A$ is given
by~\cite[theorem~3.5]{VA:THH}, because $\pi_*(A\wedge_R A^{\op})$
is separable over $A_*$ and hence the possible obstructions to an
$A_\infty$-structure on $A$ (which live in Hochschild cohomology
groups of $\pi_*(A\wedge_R A^{\op})$ over $A_*$) are trivial. The
possibility of rigidification follows from~\cite[II.4]{EKMM}.
Uniqueness also follows from the vanishing of all higher Hochschild
cohomology groups.

As $A_0$ is finitely generated projective and faithful over $\pi_0R$,
$r(A)$ is dualizable and faithful as an $R$-module spectrum. The
Azumaya condition
\begin{equation*}
\mu \: A_0 \otimes_{\pi_0R} A_0^\op \cong \Hom_{\pi_0R}(A_0,A_0)
\end{equation*}
for $A_0$ guarantees that the $\mu$-map
\begin{equation*}
\mu \: r(A) \wedge_R r(A)^\op \lra F_R(r(A),r(A))
\end{equation*}
is a weak equivalence.
\end{proof}

\begin{cor}\label{cor:real}
There is a natural group homomorphism
\begin{equation*}
r\:\Br(\pi_0R)\lra \Br(R);\quad [A]\mapsto [r(A)].
\end{equation*}
\end{cor}

This result implies Proposition \ref{prop:real0} but reaches
further. For instance in the presence of enough roots of one,
we can build generalized quaternionic extensions of ring spectra
or consider cyclic extensions. Note, however, that in many cases
$\Br(\pi_0(R))=0$, for instance if $\pi_0(R)$ is isomorphic to
a finite field, $\Z$ (see~\cite{Gr}) or $\Z[A]$ for some finite
abelian group $A$ (see~\cite{KO}). We learned from David Gepner
that the Brauer group for a connective commutative $S$-algebra~$R$
can be described via the second \'etale cohomology group of
$\pi_0(R)$ with coefficients in the units and the first \'etale
cohomology group of $\pi_0(R)$ with coefficients in $\Z$ via a
short exact sequence.

John Rognes drew our attention to the non-trivial examples of
Brauer groups in~\cite{RoWe}.
\begin{prop}\label{prop:JR}
There is a quaternionic extension of the sphere spectrum with
$2$ inverted that is not trivial in the Brauer group and hence
\begin{equation*}
\Br(S[1/2]) \neq 0.
\end{equation*}
\end{prop}
\begin{proof}
The Brauer group of $\mathbb{Z}[1/2]$ is isomorphic to $\Z/2$,
see~\cite[2.8]{RoWe}. Extension of scalars to $\R$ has to send
the generator of $\Br(\Z[1/2])$ to the generator of
$\Br(\R)\cong \Z/2$ and hence this extension is equivalent to
the $\R$-algebra of quaternions $\H$. Thus we know that a
representative of the generator of $\Br(\Z[1/2])$ is given by
the class of the subring of localized Hurwitz quaternions
$\Z[1/2](i,j,k)\subseteq\H$ with $i^2=j^2=k^2=-1$, $k=ij=-ji$,
which is the quaternionic extension of $\Z[1/2]$. We can realize
this extension topologically as an Azumaya algebra $H$ over
$S[1/2]$.

Using base-change to $H\R$ we get the following commutative
diagram.
\begin{equation*}
\xymatrix{
{\Br(S[1/2])} \ar[r] & {\Br(H\R)} \\
{\Br(\Z[1/2])} \ar[r] \ar[u]& {\Br(\R)}
\ar[u]
}
\end{equation*}
As the map from $\Br(\mathbb{R})$ to $\Br(H\R)$ is injective
(compare Remark~\ref{rem:inj}), the image of the class of
$\Z[1/2](i,j,k)$ in $\Br(S[1/2])$ cannot be trivial.
\end{proof}
\begin{rem} \label{rem:localspheres}
As the Brauer groups of $\Z[1/p]$ and $\Z_{(p)}$ are non-trivial
for odd primes as well (see \cite{OS} and \cite[p.145]{Weibel}),
the above result can be used to obtain that other Brauer groups
of connective commutative ring spectra are non-trivial. In particular,
the Brauer groups of the corresponding localized spheres are non-trivial.
\end{rem}

\section{Topological Hochschild cohomology of group rings}
\label{sec:grouprings}

We will consider Azumaya algebra extensions that arise as follows.
For a finite discrete group $G$ and a commutative $S$-algebra $A$,
we consider the group $A$-algebra spectrum  $A[G] = A \wedge G_+$.
Note, that if $G$ is not abelian, then $A[G]$ is not commutative.
We want to identify the extension $\THH_A(A[G]) \lra A[G]$ as an
Azumaya extension in good cases.

For ordinary commutative rings $R$ and  groups $G$, DeMeyer and
Janusz describe in~\cite{DMJ} conditions on $R$ and $G$ which
ensure that $R[G]$ is an Azumaya algebra over its centre. First,
we document a well-known identification of topological Hochschild
cohomology of group rings, see for instance~\cite[4.2.10]{Malm}.
This can we viewed as a topological version of Mac~Lane's
isomorphisms~\cite[7.4.2]{Loday}.

\begin{lem} \label{lem:thhag}
For $A$ and $G$ as above we have
\begin{equation*}
\THH_A(A[G],A[G]) \simeq (A[G]^c)^{hG} = F_G(EG_+,A[G]^c).
\end{equation*}
Here $A[G]^c$ denotes the naive $G$-spectrum $A[G]$, where~$G$
acts by conjugation on $G$.
\end{lem}
\begin{proof}
Topological Hochschild cohomology of $A[G]$ can be described
as the totalization of the cosimplicial spectrum that has
\begin{equation*}
F_A(A[G]^q,A[G]) \cong F(G^q_+,A[G])
\end{equation*}
as $q$-cosimplices~\cite{MS}. First, we mimic the identification
that is used in the Mac~Lane isomorphism for usual Hochschild
cohomology in order to identify this cosimplicial spectrum with
the one that has $F(G^q,A[G]^c)$ as $q$-cosimplices. In algebra
this identification is given by $f \mapsto f'$ where
\begin{equation*}
f'(g_1,\ldots,g_q) = f(g_1,\ldots,g_q)g_q^{-1}\ldots g_1^{-1}.
\end{equation*}
An analogous identification works on spectrum level. The coface
maps in the cosimplicial structure in $F(G^\bullet,A[G]^c)$ are
given by
\begin{align*}
d_0(f)(g_1,\ldots,g_q) & =  g_1f(g_2,\ldots,g_q)g_1^{-1}, \\
d_i(f)(g_1,\ldots,g_q) & =  f(g_1,\ldots, g_ig_{i+1},\ldots,g_q),
\quad (0 < i < q) \\
d_q(f)(g_1,\ldots,g_q  & =  f(g_1,\ldots,g_{q-1}).
\end{align*}

Consider the simplicial model of $EG$ with $q$-simplices $G^{q+1}$,
with diagonal $G$-action, and where the $i$-th face map in $EG$
is given by  omitting the $i$-th group element. We can write the
homotopy fixed point spectrum $F_G(EG_+,A[G]^c)$ as
\begin{equation*}
F_G(EG_+,A[G]^c) \cong \Tot([q] \mapsto F_G(G^{q+1},A[G]^c)).
\end{equation*}

Let $\phi\: F(G^\bullet,A[G]^c) \lra F_G(EG_+,A[G]^c)$ be the
map that we can describe symbolically as
\begin{equation*}
(\phi f)(g_0,\ldots,g_q) = g_0f(g_0^{-1}g_1,\ldots,g_{q-1}^{-1}gq)g_0^{-1}.
\end{equation*}
It is then straightforward to check that $\phi$ in fact respects
the cosimplicial structure.
\end{proof}

Now fix a prime $p$. Let $\k$ be an algebraically closed field
of characteristic~$p$ and let $H\k$ be the corresponding
Eilenberg-Mac~Lane spectrum realized as a commutative $S$-algebra.
We also adopt the notation of~\cite{AB&BR:E_n^nr}. Thus $E_n$
is the Lubin-Tate spectrum associated with the prime~$p$ and the
Honda formal group of height~$n$ and $\Enr_n$ is its maximal
unramified Galois extension. These commutative $S$-algebras have
`residue fields' in the sense of~\cite{AB&BR:Galois,AB&BR:Inv},
namely $K_n$ and $\Knr_n$ respectively, and these are algebras
over $E_n$ and $\Enr_n$ respectively, but only homotopy commutative
when $p\neq 2$ and not even that when $p=2$.

\begin{thm} \label{thm:grouprings}
Let $G$ be a non-trivial finite discrete group whose order is not
divisible by~$p$. Suppose that $A$ is either $H\k$ or $\Enr_n$.
\begin{enumerate}
\item
If\/ $G$ is abelian, then $(A[G]^c)^{hG} \lra A[G]$ and the trivial
extension $\id \: A[G] \lra A[G]$ are equivalent.
\item
If $G$ is non-abelian, then $A[G]$ is a non-trivial
$(A[G])^{hG}$-Azumaya algebra.
\end{enumerate}
\end{thm}
\begin{proof}
In all cases, we will consider the homotopy fixed point spectral
sequence
\begin{equation*}
\mathrm{E}^{s,t}_2 = H^{-s}(G;A_t[G]^c)
                         \Lra \pi_{s+t}((A[G]^c)^{hG}).
\end{equation*}
If $p$ does not divide the order of the group $G$, then this
spectral sequence collapses and the only surviving non-trivial
terms are the $G$-invariants
\begin{equation*}
\mathrm{E}_2^{0,t} = (A_t[G]^c)^G
\end{equation*}
which can be identified with the center of the group ring $Z(A_*[G])$.
In particular, $\pi_*((A[G]^c)^{hG})$ is a graded commutative
$A_*$-algebra.

If $G$ is abelian, then the conjugation action is trivial and
as~$p$ does not divide~$|G|$ we obtain
\begin{equation*}
(A[G]^c)^{hG} = F(BG_+,A[G]) \simeq A[G],
\end{equation*}
so we have the trivial Azumaya extension. If $G$ is not abelian,
then the center of the group ring $A_*[G]$ is a proper subring
of~$A_*[G]$.

For $A = H\k$ we can use Artin-Wedderburn theory to obtain a
splitting of the semisimple ring $\k[G]$ into a product of
matrix algebras over the algebraically closed field $\k$,
\begin{equation*}
\k[G] \cong \prod_{i=1}^r M_{m_i}(\k),
\end{equation*}
where $r$ agrees with the number of conjugacy classes in $G$.
Thus the center of $\k[G]$ is a product of copies of $\k$ and
is therefore an \'etale $\k$-algebra. By the obstruction
theory of Robinson or Goerss-Hopkins~\cite{AR:Gamma,GH}, there
is a unique $E_\infty$ $H\k$-algebra spectrum that is weakly
equivalent to $(A[G]^c)^{hG}$. By abuse of notation we denote
the corresponding commutative $H\k$-algebra by $(A[G]^c)^{hG}$.

We have to describe $A[G]$ as an associative $(A[G]^c)^{hG}$-algebra.
For this we use~\cite[theorem~3.5]{VA:THH} again. Starting with our
commutative model of $(A[G]^c)^{hG}$ we can build a homotopy associative
ring spectrum $B$ with $\pi_*(B)\cong A_*[G]$, and as~$G$ is finite
and discrete this extension is of the form
\begin{equation*}
\pi_*(B) \cong
          \pi_*(A[G]^c)^{hG} \otimes_{\pi_0(A[G]^c)^{hG}} B_0,
\end{equation*}
with $\pi_0(A[G]^c)^{hG} \lra B_0$ being algebraically Azumaya.
Thus we can apply Theorem~\ref{thm:real} to see that there is
an associative $(A[G]^c)^{hG}$-algebra $B$ which models $A[G]$
and such that $B$ is Azumaya over $(A[G]^c)^{hG}$.

For $\Enr_n$ we pass to the residue field $\Knr_n$. The homotopy
fixed point spectral sequence gives
\begin{align*}
\pi_*((\Enr_n[G]^c)^{hG}) &\cong Z((\Enr_n)_*[G]) \\
        &\cong Z(W\bar{\F}_p[[u_1,\ldots,u_{n-1}]][G])[u^{\pm1}].
\end{align*}
Reducing modulo the maximal ideal $\mathfrak{m} = (p,u_1,\ldots,u_{n-1})$
gives the homotopy groups of the $G$-homotopy fixed points of $\Knr_n[G]$
with respect to the conjugation action, $Z(\bar{\F}_p[G])[u^{\pm 1}]$ and
again we can identify this term as $\prod_{i=1}^r \bar{\F}_p$ where~$r$
denotes the number of conjugacy classes in $G$. The idempotents that give
rise to these splittings can be lifted to idempotents for
$Z(W\bar{\F}_p[[u_1,\ldots,u_{n-1}]][G])$ and
$W\bar{\F}_p[[u_1,\ldots,u_{n-1}]][G]$ and therefore these two algebras
also split into products with $r$ factors:
\begin{align*}
W\bar{\F}_p[[u_1,\ldots,u_{n-1}]][G] &\cong \prod_{i=1}^r B_i, \\
Z(W\bar{\F}_p[[u_1,\ldots,u_{n-1}]][G]) &\cong \prod_{i=1}^r C_i,
\end{align*}
where
\begin{equation*}
B_i/\mathfrak{m} \cong M_{m_i}(\bar{\F}_p),
\end{equation*}
while for $1 \leq i \leq r$, the $C_i$ are commutative and satisfy
\begin{equation*}
C_i/\mathfrak{m}C_i \cong \bar{\F}_p.
\end{equation*}

Additively we know that $Z(W\bar{\F}_p[[u_1,\ldots,u_{n-1}]][G])$ is
the free module on the conjugacy classes and so we can conclude that
$(\Enr_n[G]^c)^{hG}$ is weakly equivalent to $\prod_{i=1}^r\Enr_n$ and
the latter spectrum can be modelled by a commutative $\Enr_n$-algebra
spectrum and $\Enr_n[G]$ is dualizable over $\prod_{i=1}^r\Enr_n$.

Artin-Wedderburn theory gives a semisimple decomposition
\begin{equation*}
\bar{\F}_p[G] \iso \prod_{i=1}^r M_{d_i}(\bar{\F}_p),
\end{equation*}
and the centre $Z(\bar{\F}_p[G])$ can be identified with the product
of the centres of the matrix ring factors. There are associated
central idempotents of $\bar{\F}_p[G]$ accomplishing this splitting.
By the theory of idempotent lifting described in~\cite[section~21]{Lam:Noncomm}
for example, these idempotents lift to give an associated splitting
\begin{equation*}
W\bar{\F}_p[[u_1,\ldots,u_{n-1}]][G] \iso
  \prod_{i=1}^r M_{d_i}(W\bar{\F}_p[[u_1,\ldots,u_{n-1}]]),
\end{equation*}
and again the centre of $W\bar{\F}_p[[u_1,\ldots,u_{n-1}]][G]$ can
be identified with the product of the centres of the matrix factors.
Notice that $M_{d_i}(W\bar{\F}_p[[u_1,\ldots,u_{n-1}]])$
is Azumaya over $W\bar{\F}_p[[u_1,\ldots,u_{n-1}]]$. The rest of
the proof involves realising the central idempotents as morphisms
of $S$-algebras, but this is well known to be possible since the
projections are Bousfield localisations, see~\cite{SVW}.
\end{proof}

\section{Azumaya algebras over Lubin-Tate spectra}
\label{sec:Azumaya-LT}

From now on we will use $E$ to denote $E_n$, $\Enr_n$ or any commutative
Galois extension of $E_n$ obtained as a homotopy fixed point algebra
$E=(\Enr_n)^{h\Gamma}$ for some closed normal subgroup
$\Gamma\lhd\Gal(\Fc_p/\F_{p^n})$. Similarly, $K$ will denote the
corresponding residue field of $E$, so when $E=E_n$ or $\Enr_n$
we have $K=K_n$ or $\Knr_n$.

We will work with dualizable $K$-local $E$-modules.
By~\cite[section~7]{AB&BR:E_n^nr} we know that such modules are
retracts of finite cell $E$-modules. If $W\in\mathscr{M}_{E,K}$,
then since $\pi_*(K\wedge_E W)$ is a graded vector space over
the graded field $K_*=\pi_*(K)$, it follows that
\[
K\wedge_E W \simeq L_K\bigvee_i \Sigma^{d(i)}K,
\]
where the right hand wedge is non-trivial if and only if $W$ is
non-trivial in $\mathscr{D}_{E,K}$. In particular, if $W$ is
dualizable this wedge is finite and
\[
K\wedge_E W \simeq \bigvee_i \Sigma^{d(i)}K
\]
since $W$ is $K$-local. For any $X\in\mathscr{M}_{E,K}$,
\[
K\wedge_E(W\wedge_E X) \simeq L_K\bigvee_i \Sigma^{d(i)}K\wedge_E X,
\]
so $W\wedge_E X$ is trivial in $\mathscr{D}_{E,K}$ if and only if
both of $W$ and $X$ are trivial in $\mathscr{D}_{E,K}$. Thus every
$E$-module $W$ which is non-trivial as an element of $\mathscr{D}_{E,K}$
is faithful and cofibrant as a $K$-local $E$-module; furthermore,
every $X\in\mathscr{M}_{E,K}$ is $W$-local.

By~\cite{VA:THH}, there are many examples of $K$-local Azumaya
algebras over $E$ which have $K$ as their underlying ring
spectrum. These examples have no analogue in the algebraic context
since they are not projective $E$-modules, nor do they split over
suitable Galois extensions. Instead we focus on split examples.
A good source of these can be found in the situation
of~\cite[section~5.4.3]{JR:Opusmagnus}, based on work of Devinatz
and Hopkins~\cite{ED&MH:HtpyFixPts} and we will discuss these in
Section~\ref{sec:K-localAzumaya}.

For background ideas on Azumaya algebras graded on a finite abelian
group, we follow~\cite{CGO} which generalises work of Wall~\cite{Wall}
and others. We will only consider the case where the grading group
is $\Z/2$ with the non-trivial symmetric bilinear map
$\Z/2 \times \Z/2 \lra\{\pm1\}$ determining the relevant signs, however
in periodic topological contexts it may also prove useful to modify
the grading to other finite quotient groups lying between $\Z$ and
$\Z/2$, and the above reference should provide appropriate generality
for such algebra.

Over a field $\k$, an (ungraded) Azumaya algebra $A$ is a central
simple algebra, so by Wedderburn's theorem, there is an isomorphism
of $\k$-algebras
\[
A \iso M_r(D),
\]
where $D$ is a central division algebra over $\k$. If $d=\dim_\k D$,
then
\[
\dim_\k A = (rd)^2,
\]
so $\dim_\k A $ is a square. In the graded case, such restrictions
do not always apply, and this has consequences for the topological
situation.
\begin{thm}\label{thm:KsmashA}
Suppose that $p$ is an odd prime and let $A$ be a $K$-local Azumaya
algebra over $E$. Then $\pi_*(K\wedge_E A)$ is an Azumaya algebra
over $K_*$.
\end{thm}
\begin{proof}
The ring $K_*$ is a $2$-periodic graded field which we will view
as $\Z/2$-graded, and $\pi_*(K\wedge_E A)$ will also be viewed as
a $\Z/2$-graded $K_*$-algebra.

We have isomorphisms of $K_*$-algebras
\begin{align*}
\pi_*(K\wedge_E A)\otimes_{K_*}\pi_*(K\wedge_E A)^\op
 & \iso \pi_*(K\wedge_E A)\otimes_{K_*}\pi_*(K^{\op}\wedge_E A^\op) \\
 & \iso \pi_*(K\wedge_E (A\wedge_E A^\op)) \\
 & \iso \pi_*(K\wedge_E F_E(A,A)).
\end{align*}
Since $A$ and $K$ are strongly dualizable, using results of~\cite{EKMM}
we have
\[
K\wedge_E F_E(A,A) \simeq F_K(K\wedge_EA,K\wedge_E A),
\]
so the universal coefficient spectral sequence over~$K$ yields
\[
\pi_*(K\wedge_E F_E(A,A)) \iso \End_{K_*}(\pi_*(K\wedge_E A)).
\]
Therefore $\pi_*(K\wedge_E A)$ is a $K_*$-Azumaya algebra.
\end{proof}
\begin{cor}\label{cor:KsmashA}
If $\pi_*(K\wedge_E A)$ is concentrated in even degrees then
its dimension is a square, \ie, for some natural number~$m$,
\[
\dim_{K_*}\pi_*(K\wedge_E A) = m^2.
\]
\end{cor}

In fact we have
\begin{prop}\label{prop:K*A-even}
If $\pi_*(K\wedge_E A)$ is concentrated in even degrees then
$\pi_*(A)$ is a $\Z/2$-graded algebra Azumaya algebra over
$E_0$. In particular, as an $E$-module $A$ is equivalent to
a wedge of $m^2$ copies of $E$, where
\[
m^2 = \dim_{K_*}\pi_*(K\wedge_E A) = \rank_{E_*} \pi_*(A).
\]
\end{prop}
\begin{proof}
By \cite{AB&BR:E_n^nr} (see section~7 and the proof of theorem~5.1),
the $E_*$-module $\pi_*(A)$ is finitely generated, free and
concentrated in even degrees, hence
\[
\pi_*(A)\otimes_{E_*}\pi_*(A)^\op \iso \pi_*(A\wedge_E A^\op)
   \iso \pi_*(F_E(A,A)) \iso \Hom_{E_*}(\pi_*(A),\pi_*(A)),
\]
where the last isomorphism follows from the collapsing of the
universal coefficient spectral sequence.
\end{proof}

Recall that $\Az_K(E)$ is the collection of all cofibrant
$K$-local topological Azumaya algebras over $E$. The Brauer
equivalence relation $\approx$ on $\Az_K(E)$ is then given
as follows:
\begin{itemize}
\item
If $A,B\in\Az_K(E)$, then $A\approx B$ if and only if there
are faithful, dualizable, cofibrant $E$-modules $U,V$ for
which there is an equivalence in the derived category of
$K$-local $E$-algebras
\[
A\wedge_E F_E(U,U)\simeq B\wedge_E F_E(V,V).
\]
\end{itemize}
The set of equivalence classes of $\approx$ is $\Br_K(E)$;
this is indeed a set since every dualizable $K$-local $E$-module
is a retract of a finite cell $E$-module.

\section{Some examples of $K_n$-local Azumaya algebras}
\label{sec:K-localAzumaya}

We now recall Proposition~\ref{prop:Saltman-6.11}. By work of
Devinatz and Hopkins~\cite{ED&MH:HtpyFixPts}, and subsequently
Davis~\cite{DD:iterated}, as explained
in~\cite[theorem~5.4.4]{JR:Opusmagnus}, for each pair of closed
subgroups
\begin{equation*}
H\leq G\leq\mathbb{G}_n =
                \Gal(\F_{p^n}/\F_p)\ltimes\mathbb{S}_n
\end{equation*}
of the Morava stabilizer group, there is an associated pair
of homotopy fixed point spectra $E^{hG}\lra E^{hH}$, and if
$H\lhd G$ then this is a $K$-local $G/H$-Galois extension.
In particular, when $H\leq \mathbb{G}_n$ is finite,
$E^{hH}\lra E$ is a $K$-local $H$-Galois extension.

A particularly interesting source of examples is provided by
taking $G$ to be a maximal finite subgroup of $\mathbb{G}_n$.
If $p$ is odd and $n=(p-1)k$ with $p\nmid k$, or $p=2$ and
$n=2k$ with $k$ odd, then such maximal subgroups are unique
up to conjugation and then the homotopy fixed point spectrum
$E^{hG}$ is denoted $EO_n$.
For $p=3$ Behrens~\cite[remark~1.7.3]{Be} gives an argument
for the identification of $EO_2$ with the $K(2)$-localization
of the spectrum of topological modular forms, $\mathit{TMF}$.
This can be adapted to $p=2$. We proceed with an example,
studied in~\cite[section~5.4.3]{JR:Opusmagnus}.

\begin{ex}\label{ex:EO2}
At the prime $p=2$, the group $\mathbb{G}_2$ has a maximal
finite subgroup $G_{48}$ of order~$48$ which is isomorphic
to a semi-direct product of the group $HQ_{24}$ of order $24$,
which consists of the units in the ring of Hurwitz quaternions,
with the Galois group $\Gal(\F_4/\F_2)$ of order~$2$. Therefore
$E_2/EO_2$ is a $G_{48}$-Galois extension. Applying
Proposition~\ref{prop:Saltman-6.11}, we see that the splitting
of the group $G_{48}$ yields an Azumaya algebra
\begin{equation*}
(E_2\< HQ_{24}\>)^{hC_2}
\end{equation*}
over $EO_2$. There are two more examples like this: Of the $15$
conjugacy classes of subgroups of~$G_{48}$, there are $7$ which
are normal: $1$, $C_2$, $C_4$, $Q_8$, $\Syl_2(G_{48})$, $HQ_{24}$,
and $G_{48}$. The first and last of them are uninteresting here,
as $(E_2)^{hG_{48}}\simeq EO_2$ and $E_2\<G_{48}\>\simeq F_{EO_2}(E_2,E_2)$
are trivial in the Brauer group of $EO_2$. The second and third
do not split the group, but the other three do. The last of these
yields the Azumaya algebra displayed above, but the other two give
rise to further examples
\begin{equation*}
(E_2\< Q_8\>)^{hC_6}
\quad
(E_2\< \Syl_2(G_{48})\>)^{hC_3}
\end{equation*}
of Azumaya algebras over $EO_2$.
\end{ex}

\section{The Brauer group of the $K(n)$-local
  sphere}\label{sec:Br(SK)}

In this section we discuss the $K(n)$-local Brauer group of the
$K(n)$-local sphere $L_{K(n)}S$.
\begin{thm}\label{thm:Br(SK)}
Suppose that\/ $p>2$ and\/ $n>1$. Then the $K(n)$-local Brauer
group of $L_{K(n)}S$ is non-trivial.
\end{thm}
\begin{proof}
As usual, let us write $q=p^n$. The cyclic group $C=\F_q^\times$
of order $q-1$ consists of roots of unity, and the Galois
group $G=\Gal(\F_q/F_p)$ is cyclic of order~$n$, generated by the
Frobenius. As $n\neq1$, the Galois group is non-trivial, and its
Galois action gives rise to an extension $F=G\ltimes C$.

Let $\mathbb{G}_n$ denote again the $n$-th extended Morava stabilizer
group. We refer to section~2.3 and the appendix of~\cite{Henn} for
the following facts. The reduction of the determinant gives rise
to a surjection $\mathbb{G}_n\to F$, a splitting of which is induced
by the Teichm\"uller character $\F_q^\times\lra\mathbb{W}\F_q^\times$.
Consequently, if we write $N$ for the kernel of that surjection, then
there is an isomorphism $\mathbb{G}_n\cong F\ltimes N$.

We will now invoke Proposition~\ref{prop:Saltman-6.11} in order
to get an Azumaya algebra
\begin{equation*}
A = (E_n^{hN}\< C\>)^{hG}
\end{equation*}
over $L_{K(n)}S$. That result also implies that the image of $[A]$
automatically maps to zero in the local Brauer group of $E_n^{hN}$.
In particular,
it vanishes in $\Br_{K(n)}(E_n)$ itself. It remains to show that
$[A]\neq0$ in the local Brauer group of $L_{K(n)}S$. In particular,
it suffices to prove that its image in  $\Br_{K(n)}(E_n^{hF})$ is
non-zero. That image is equivalent to $(E_n\< C\>)^{hG}$.

We assume on the contrary that there were an equivalence between
our example $(E_n\<C\>)^{hG}$ and $F_{E_n^{hF}}(W,W)$ for some
faithful, dualizable, cofibrant $E_n^{hF}$-module~$W$. We get a
contradiction by looking at the centers of $\pi_0\otimes\Q$ for
both algebras.

The center of $\pi_0(F_{E_n^{hF}}(W,W))\otimes\Q$ is just
\begin{equation*}
\pi_0(E_n^{hF})\otimes\Q\cong\pi_0(E_n\otimes\Q)^{F}.
\end{equation*}
As the group of roots of unity $C$ acts only on the grading, and
the Galois group $G$ acts only on the coefficients, this is isomorphic
to $\Q_p[\![u_1,\dots,u_{n-1}]\!]$. However, as~$p$ is odd, the
$G$-action on $E_n\<C\>$ not only leaves the summand corresponding
to the root~$1$ in $C$ invariant, but also the one corresponding
to the unique element $-1$ of order $2$ in $C$. As a consequence,
$\pi_0((E_n\< C\>)^{hG})\otimes\Q$ contains more than
$\Q_p[\![u_1,\dots,u_{n-1}]\!]$.
\end{proof}
Summarizing, we know that localizations of the sphere (at an ordinary
prime or at Morava-$K$-theory) possess non-trivial Brauer groups,
compare Remark~\ref{rem:localspheres}. However, we conjectured that
the Brauer group of the (global) sphere spectrum is trivial.
Antieau-Gepner's \cite{AG}, Gepner-Lawson's~\cite{GL} and To\"en's~\cite{To}
results that relate Brauer groups of commutative $S$-algebras to
\'etale cohomology groups, allow to prove this conjecture, so
\begin{equation*}
\Br(S) = 0.
\end{equation*}

\end{document}